\documentclass[12pt,reqno]{amsart}
\usepackage{amsthm}
\usepackage{amssymb}
\usepackage{graphics}
\usepackage{tikz}
\usetikzlibrary{shapes,backgrounds,calc}
\usepackage{latexsym}
\usepackage{multicol}
\usepackage{verbatim,enumerate}
\usepackage{accents}
\usepackage{cite}
\usepackage{multirow}
\usepackage{bigstrut}
\usepackage{amsthm}
\usepackage{amssymb}
\usepackage{graphics}
\usepackage{tikz}
\usetikzlibrary{shapes,backgrounds,calc}
\usepackage{latexsym}
\usepackage{multicol}
\usepackage{verbatim,enumerate}
\usepackage{accents}
\usepackage{cite}
\usepackage{array}
\usepackage[colorlinks=true, linkcolor=blue, citecolor=blue, urlcolor=black]{hyperref}
\usepackage{hyperref}
\usepackage{amsmath, amscd,url}
\usepackage{multirow}
\usepackage{bigstrut}
\usepackage{longtable}

\usepackage{stackengine}

\newcommand{\Z}{\mbox{$\mathbb Z$}}	
\newcommand{\Q}{\mathbb Q}	

\usepackage{hyperref}
\usepackage{amsmath, amscd,url}
\usepackage{longtable}

\advance\textwidth by 1.3in \advance\oddsidemargin by -.6in \advance\evensidemargin by -.6in
\theoremstyle{definition}

\newtheorem{theorem}{Theorem}[section]
\newtheorem{lemma}[theorem]{Lemma}

\newtheorem{corollary}[theorem]{Corollary}
\theoremstyle{definition}

\newtheorem{example}[theorem]{Example}

\theoremstyle{remark}
\newtheorem{remark}[theorem]{Remark}

\theoremstyle{definition}

\newcounter{cnt}
 \makeatletter
\def\mydggeometry{\makeatletter\dg@YGRID=1\dg@XGRID=20\unitlength=0.003pt\makeatother}
\makeatother \theoremstyle{remark}

\numberwithin{equation}{section}
\let\bwdg\bigwedge
\def\bigwedge{{\textstyle\bwdg}}

\newcommand{\nc}{\newcommand}
\newcommand{\rnc}{\renewcommand}

\nc{\cal}{\mathcal} \nc{\goth}{\mathfrak} \rnc{\bold}{\mathbf}

\nc\bomega{{\mbox{\boldmath $\omega$}}} \nc\bpsi{{\mbox{\boldmath $\Psi$}}}
 \nc\balpha{{\mbox{\boldmath $\alpha$}}}
 \nc\bpi{{\mbox{\boldmath $\pi$}}}
 \nc\bvpi{{\mbox{\boldmath $\varpi$}}}
\nc\chara{\operatorname{ch}}

  \nc\bxi{{\mbox{\boldmath $\xi$}}}
\nc\bmu{{\mbox{\boldmath $\mu$}}} \nc\bcN{{\mbox{\boldmath $\cal{N}$}}} \nc\bcm{{\mbox{\boldmath $\cal{M}$}}} \nc\blambda{{\mbox{\boldmath
$\lambda$}}}\nc\bnu{{\mbox{\boldmath $\nu$}}}

\makeatletter
\def\section{\def\@secnumfont{\mdseries}\@startsection{section}{1}%
  \z@{.7\linespacing\@plus\linespacing}{.5\linespacing}%
  {\normalfont\scshape\centering}}
\def\subsection{\def\@secnumfont{\bfseries}\@startsection{subsection}{2}%
  {\parindent}{.5\linespacing\@plus.7\linespacing}{-.5em}%
  {\normalfont\bfseries}}
\makeatother

 \nc{\Hom}{\operatorname{Hom}}
  \nc{\mode}{\operatorname{mod}}
\nc{\End}{\operatorname{End}} \nc{\wh}[1]{\widehat{#1}} \nc{\Ext}{\operatorname{Ext}} \nc{\ch}{\text{ch}} \nc{\ev}{\operatorname{ev}}
\nc{\Ob}{\operatorname{Ob}} \nc{\soc}{\operatorname{soc}} \nc{\rad}{\operatorname{rad}} \nc{\head}{\operatorname{head}}

 \nc{\Cal}{\cal} \nc{\Xp}[1]{X^+(#1)} \nc{\Xm}[1]{X^-(#1)}

\nc{\N}{{\bold N}}  \nc\boa{\bold a} \nc\bob{\bold b} \nc\boc{\bold c} \nc\bod{\bold d} \nc\boe{\bold e} \nc\bof{\bold f} \nc\bog{\bold g}
\nc\boh{\bold h} \nc\boi{\bold i} \nc\boj{\bold j} \nc\bok{\bold k} \nc\bol{\bold l} \nc\bom{\bold m} \nc\bon{\mathbb n} \nc\boo{\bold o}
\nc\bop{\bold p} \nc\boq{\bold q} \nc\bor{\bold r} \nc\bos{\bold s} \nc\boT{\bold t} \nc\boF{\bold F} \nc\bou{\bold u} \nc\bov{\bold v}
\nc\bow{\bold w} \nc\boz{\bold z}\nc\ba{\bold A} \nc\bb{\bold B} \nc\bc{\mathbb C} \nc\bd{\bold D} \nc\be{\bold E} \nc\bg{\bold
G} \nc\bh{\bold H} \nc\bi{\bold I} \nc\bj{\bold J} \nc\bk{\bold K} \nc\bl{\bold L} \nc\bm{\bold M} \nc\bn{\mathbb N} \nc\bo{\bold O} \nc\bp{\bold
P} \nc\bq{\bold Q} \nc\br{\bold R} \nc\bs{\bold S} \nc\bt{\bold T} \nc\bu{\bold U} \nc\bv{\bold V} \nc\bw{\bold W} \nc\bz{\mathbb Z} \nc\bx{\bold
x} \nc\KR{\bold{KR}} \nc\rk{\bold{rk}} \nc\het{\text{ht }}

\nc\toa{\tilde a} \nc\tob{\tilde b} \nc\toc{\tilde c} \nc\tod{\tilde d} \nc\toe{\tilde e} \nc\tof{\tilde f} \nc\tog{\tilde g} \nc\toh{\tilde h}
\nc\toi{\tilde i} \nc\toj{\tilde j} \nc\tok{\tilde k} \nc\tol{\tilde l} \nc\tom{\tilde m} \nc\ton{\tilde n} \nc\too{\tilde o} \nc\toq{\tilde q}
\nc\tor{\tilde r} \nc\tos{\tilde s} \nc\toT{\tilde t} \nc\tou{\tilde u} \nc\tov{\tilde v} \nc\tow{\tilde w} \nc\toz{\tilde z} \nc\woi{w_{\omega_i}}

\begin{document}


\title[Monogenic Binomial Composition]{A Study of monogenity of Binomial Composition}

\author[Anuj Jakhar]{Anuj Jakhar}
	\author[Ravi Kalwaniya]{Ravi Kalwaniya}
	\author[Prabhakar yadav]{Prabhakar yadav}
	\address[Anuj Jakhar, Ravi Kalwaniya]{Department of Mathematics, Indian Institute of Technology (IIT) Madras}
	\address[Prabhakar Yadav]{Indian Statistical Institute, New Delhi}
	\email[Anuj Jakhar]{anujjakhar@iitm.ac.in\\}
	\email[Ravi Kalwaniya]{ma22d021@smail.iitm.ac.in}
	\email[Prabhakar yadav]{pkyadav914@gmail.com}




\maketitle

\vspace{-0.2in}

\begin{abstract}
Let $\theta$ be a root of a monic polynomial $h(x) \in \Z[x]$ of degree $n \geq 2$. We say $h(x)$ is monogenic if it is irreducible over $\Q$ and $\{ 1, \theta, \theta^2, \ldots, \theta^{n-1} \}$ is a basis for the ring $\Z_K$ of integers of $K = \Q(\theta)$. In this article, we study about the monogenity of number fields generated by a root of composition of two binomials. We characterise all the primes dividing the index of the subgroup $\Z[\theta]$ in $\Z_K$ where $K = \Q(\theta)$ with $\theta$ having minimal polynomial  $F(x) = (x^m-b)^n - a \in \Z[x]$, $m\geq 1$ and $n \geq 2$. As an application, we provide a class of pairs of binomials $f(x)=x^n-a$ and $g(x)=x^m-b$ having the property that both $f(x)$ and $f(g(x))$ are monogenic. 

\end{abstract}
\maketitle

\section{Introduction and statements of results}\label{intro}

Let $K = \Q(\theta)$ be an algebraic number field with $\theta$ in the ring $\Z_K$ of algebraic integers of $K$ and $h(x)$ having degree $n$ be the minimal polynomial of $\theta$ over the field $\Q$ of rational numbers.
Let $d_K$ denote the discriminant of $K$ and $D_h$ denote the discriminant of the polynomial $h(x)$. It is well-known that $d_K$ and $D_h$ are related by the following formula $$D_h = [\Z_K : \Z[\theta]]^2d_K.$$ We say that $h(x)$ is monogenic if $\Z_K = \Z[\theta]$, or equivalently, if $D_h = d_K$. In this case, $\{1, \theta, \cdots, \theta^{n-1}\}$ will be an integral basis of $K$ and $K$ will be a monogenic number field. A number field $K$ is called monogenic if there exists some $\alpha \in \Z_K$ such that $\Z_K = \Z[\alpha]$. Note that the monogenity of a polynomial $h(x)$ implies that $K= \Q(\theta)$ is monogenic where $\theta$ is a root of $h(x)$, but the converse is not true in general. For example, let $\alpha$ and $\beta$ be the roots of $h_1(x)=x^2-5$ and $h_2(x)=x^2-x-1$ respectively, then $\Q(\alpha) = \Q(\beta)$. Although $h_2(x)$ is monogenic  but $h_1(x)$ is not monogenic.

The determination of monogenity of an algebraic number field is one of the classical and important problems in algebraic number theory. An arithmetic characterisation of monogenic number fields is a problem due to Hasse (cf. \cite{Hasse}). Ga{\'a}l's book provides some classification of monogenity in lower degree number fields (cf. \cite{Gaal}). In 2016, Jhorar and Khanduja gave necessary and sufficient conditions for $\Z_K = \Z[\theta]$ when $\theta$ is a root of an irreducible binomial $x^n-b \in \Z[x]$ (cf. \cite[Theorem 1.3]{Jh-Kh}. In fact, they proved the following result.
\begin{theorem}\label{binom}
	Let $K = \Q(\theta)$ be a Kummer extension of $\Q$ with $\theta$ satisfying an irreducible polynomial $x^n-b$ over $\Z$. Then the following statements are equivalent:
	\begin{itemize}
		\item[(i)] $\Z_K = \Z[\theta]$.
		\item[(ii)] $b$ is square-free integer and whenever a prime $p$ divides $n$, then $p^2$ does not divide $b^{p} - b$.
	\end{itemize}
\end{theorem}
  \noindent Using Dedekind's Index Criterion, in 2016-17, Jakhar, Khanduja and Sangwan \cite{IJNT} extended the above result and provided necessary and sufficient conditions for $\Z_K = \Z[\theta]$ when $\theta$ is a root of an irreducible trinomial $x^n+ax^m+b \in \Z[x]$ having degree $n$. As an application, they provided infinitely many monogenic trinomials. 
Further in 2020, Jones and Harrington \cite{JonesHarr} investigated few pairs of binomials $f(x)= x^n-a$ and $g(x)= x^m-b$ having the property that both $f(x)$ and $f(g(x))$ are monogenic. In fact, they provided necessary and sufficient conditions for the monogenicity of the polynomial $(x^m-b)^n-a$ with $n=2,3.$  In 2022, Ga{\'a}l \cite{Gaal2} described monogenity properties of a class of binomial compositions of
  degree six. 

Let $K = \Q(\theta)$ be an algebraic number field where $\theta$ is a root of an irreducible polynomial $f(g(x))=(x^m-b)^n-a$ over $\Q$.
In this article, we characterise all the primes dividing the index $[\Z_K : \Z[\theta]]$. As an application, we provide necessary and sufficient conditions for $\Z_K = \Z[\theta].$ Moreover, we also give class of binomials $f(x)$ and $g(x)$ such that both $f(x)$, $f(g(x))$ are monogenic.

Throughout the paper, $f(x)=x^n-a$ and $g(x)= x^m-b$ will be polynomials in $\Z[x]$ having degree $n \geq 2$ and $ m \geq 1$, respectively.  $D_F$ will stand for the discriminant of an irreducible polynomial $F(x)= f(g(x)) = (x^m-b)^n-a$. Although the formula for $|D_F|$ is given in \cite[Lemma 3.1]{JonesHarr}, for the sake of completion we provide a quick derivation of $D_F$ in Lemma \ref{disc lemma} using a different approach. We prove that $D_F$ is given by 
\begin{equation} \label {eq:1.1} 
	D_F = (-1)^{\frac{n(n-1)}{2}+m} (mn)^{mn}a^{m(n-1)}((-b)^n-a)^{m-1}.
\end{equation}
For an integer $z$, $\rad(z)$ will denote the product of distinct primes dividing $z$.

We wish to point out here that in view of Theorem \ref{binom} we can consider $n\geq 2$ while studying the monogenity of the polynomial $f(g(x))$, the composition of two polynomials $f(x) = x^n-a$ and $g(x) = x^m-b$.

 Precisely stated, we prove the following result.

\begin{theorem} \label{1.1}
{\it Let $K=\mathbb \Q(\theta)$ be an algebraic number field with $\theta$ in the ring $\Z_K$ of algebraic integers of $K$ having  minimal polynomial  $F(x)=(x^m-b)^n-a$ over $\Q$ with $ m \geq 1$ and $n \geq 2$. A prime factor $p$ of the discriminant  $D_F$ of $F(x)$ does not divide $[\Z_{K} : \Z[\theta]]$ if and only if $p$  satisfies one of the following conditions:
\begin{enumerate}[(i)]
	\item when $p \mid a$, then $p^2 \nmid a$.
	\item when $p\nmid a$ and $p\mid b$ with $m=p^js,~n=p^ks',~j+k\geq 1$ and $p\nmid ss'$, then the polynomials $\frac{1}{p}[a^{p^{j+k}}-a-nb g(x)^{n-1}]$ and $x^{ss'}-a $ are coprime $modulo ~p$.
	\item when $p \nmid ab$ and $p|n$ with $k \geq 1$ as the highest power of $p$ dividing $n$, then $p^2\nmid(a^{p^k}-a)$.
	\item when $p\nmid abn$ and $p\mid m$ with $m=p^js$, then $(x^s-b)^n-a$ must be coprime to $\frac{1}{p}\left[ a^{p^j} -a + n \sum_{i=1}^{p-1} {p^j \choose ip^{j-1}} (x^s-b)^{np^j-ip^{j-1}} b^i + n(b^{p^j}-b)  \right]$.
	\item when $p\nmid abmn$, then $p^2\nmid ((-b)^n-a)$. 
\end{enumerate} }
\end{theorem}
 
 In the special case, the above theorem quickly yields the following simple result regarding monogenic polynomials.
 \begin{corollary}\label{newcor}
 	Let $K = \Q(\theta)$ and $F(x)=(x^m-b)^n-a$ be as in Theorem \ref{1.1}. Assume that $\rad(mn)$ divides $\rad(a)$. Then $\Z_K = \Z[\theta]$ if and only if both $a$ and $(-b)^n-a$ are square-free.
 \end{corollary}

 \begin{remark}
 	Note that if $p$ is prime number such that $p\nmid a$ with $a\in \Z$, then for any positive integer $s$ the exact power of $p$ dividing $a^{p^s-1}-1$ will be same as the exact power of $p$ dividing $a^{p-1}-1$; this can be easily verified keeping in mind that $p^s-1 = (p-1)m$ with $m\equiv ~1~(mod~p)$ and $a^{p-1} \equiv 1~(mod~p).$
 \end{remark}
 

Keeping in mind the above remark, the following corollary is an immediate consequence of Theorem \ref{1.1}. 
\begin{corollary} \label{cor:1.2}
Let $K = \Q(\theta)$ with $\theta$ having minimal polynomial $F(x)=(x^m-b)^n-a \in \Z[x]$ with $m \geq 1, n\geq 2$. Let $\rad(m)$ divide $\rad(an)$.  Then $\Z_{K} = \Z[\theta]$ if and only if each prime $p$ dividing $D_F$ satisfies one of the following:
\begin{enumerate}[(i)]
	\item when $p \mid a$, then $p^2 \nmid a$.
	\item when $p \nmid a$, then $p^2 \nmid (a^p-a)$.
	\item when $p\nmid abmn$, then $p^2\nmid ((-b)^n-a)$.
\end{enumerate}
\end{corollary}
Using Theorem \ref{binom} and the above corollary, we obtain the following result which extends \cite[Theorem 1.3]{JonesHarr} as well as it provides necessary and sufficient conditions for the monogenity of both $f(x)$ and $f(g(x))$. 
\begin{corollary} \label{cor:1.3}
	Let $f(x) = x^n-a$ and $F(x)=(x^m-b)^n-a \in \Z[x]$ with $m \geq 1,~n \geq 2$ and $\rad(m)$ divide $\rad(an)$. Suppose $f(x)$ and $F(x)$ are irreducible. Then $f(x)$ and $F(x)$ are monogenic if and only if the following hold:
	\begin{enumerate}[(i)]
		\item $a$ is squarefree.
		\item $a^p \not\equiv a $ (mod $p^2$) for all primes $p$ dividing $n$.
		\item For all primes $p$ dividing $((-b)^n-a)$ with $p\nmid abn$, we have $p^2\nmid ((-b)^n-a)$.
	\end{enumerate}
\end{corollary}

The following example is an application of Corollary $\ref{newcor}$. In this example, $K=\Q(\theta)$ with $\theta$ a root of $F(x)$.
\begin{example}
	Let $p$ be a prime number. Take $m=n=a=p$ and $b=2p$ in Corollary $\ref{newcor}$. So we have $F(x)=(x^p-2p)^p-p$. Note that  $|D_F|=p^{3p^2-p}((-2p)^p-p)^{p-1}$.  Therefore in view of Corollary $\ref{newcor}$, $\Z_K = \Z[\theta]$ if and only if $(-2p)^p-p$ is square-free.  It can be easily checked that $(-2p)^p-p$ is square-free for $p<43$ except for $p=11,29$. Hence $\Z_K=\Z[\theta]$ for $p<43$ except $p=11,29$. Here, all computations for checking the factorisation of discriminant were done using sage.

\end{example} 

\section{Preliminaries}
Throughout the paper, $p$ will denote a prime number and for a polynomial $h(x)$ belonging to $ \Z[x]$, we shall denote by $\bar{h}(x)$ the polynomial over $\Z/p\Z$ obtained on replacing each coefficient of $h(x)$ modulo $p$.  


We now state the following two results (see \cite[Lemma 2.6]{SKM} for Lemma \ref{2.1lemma} and   \cite[Theorem 1.19]{SKM} for Lemma $\ref{extn}$) regarding Discriminant and Norm which will be used for obtaining the formula for the discriminant of composition of two binomials.

\begin{lemma}\label{2.1lemma} Let $f(x)\in\Z[x]$ be monic and irreducible with $\deg(f)=n.$ Let $\theta$ be a root of $f(x)$ and $K=\Q(\theta).$ Then
	$$ D_f=(-1)^{n(n-1)/2}~\mathcal{N}_{K/\Q} (f'(\theta)).$$
\end{lemma}
 
\begin{lemma}\label{extn}
Let $K/F$ be an extension of degree $n$ and $\alpha$  be an element of $K$ with $[F(\alpha):F]=d$. Then $$ \mathcal{N}_{K/F}(\alpha)=\left( \mathcal{N}_{F(\alpha)/F}(\alpha)\right)^{n/d}. $$
 \end{lemma}
 \begin{lemma} \label{disc lemma}
 	Let $f(x)=x^n -a$ and $g(x)=x^m-b$ be polynomials in $\Z[x]$. Let $F(x)=f(g(x))$ be an  irreducible polynomial, then the discriminant $D_F$ of $F$ is given by
 	\begin{equation} \label{discriminant}
 		D_F= (-1)^{\frac{n(n-1)}{2}+m}(nm)^{nm} a^{m(n-1)} ((-b)^n-a)^{m-1}.
 	\end{equation}
 \end{lemma}
 \begin{proof}
 	Let $\theta$ be a root of $F(x)$. Since $F'(x)=nmx^{m-1}(x^m-b)^{n-1}$, we have 
 	\begin{equation}\label{2.1}
 	F'(\theta)=nm\theta^{m-1} (\theta^m-b)^{n-1}. 
 	\end{equation}
 We simply write $\mathcal{N}$ for the norm $\mathcal{N}_{K/\Q}$, where $K=\Q(\theta)$. Since $\mathcal{N}(\theta)=(-1)^{mn}((-b)^n-a)$ and $\mathcal{N}(mn)=(mn)^{mn}$, taking norm both sides of $\eqref{2.1}$, we obtain
\begin{equation}\label{2.3}
	\mathcal{N}(F'(\theta))=(-1)^{mn}(mn)^{mn}\mathcal{N}(\theta^m-b)^{n-1}((-b)^n-a)^{m-1}.
\end{equation}
To calculate $\mathcal{N}(\theta^m-b)$, let $\theta^m-b=z$.
Since $F(\theta)=0$, we obtain
\begin{align*}
	F(\theta)=(\theta^m-b)^n-a=z^n-a=0.
\end{align*}
Therefore, we see that $z$ is root of the polynomial $h(x)=x^n-a$. We claim that $h(x)$ is a minimal polynomial for $z$. Keeping in mind that $\Q(z)=\Q(\theta^m)$, for proving our claim, it is sufficient to show that $h(x)$ is irreducible polynomial over $\Q.$  Suppose that $h(x)=h_1(x)h_2(x)$ is the factorisation of $h(x)$ in $\Z[x]$ with $\deg h_1(x), \deg h_2(x)\geq 1$. Clearly $h(x^m-b)=F(x)$. Using this we obtain $F(x)=h_1(x^m-b)~h_2(x^m-b)$, which contradicts the irreducibility  of $F(x)$. Hence $h(x)$ is irreducible. This proves our claim. 
So we have $\mathcal{N}_{\Q(z)/\Q}(z)=(-1)^{n}(-a)$.
Using Lemma $\eqref{extn}$, we obtain
\begin{equation}\label{3.10}
	\mathcal{N}(z)=(-1)^{m(n+1)}a^m.
\end{equation}
Therefore, the theorem follows from Lemma $\ref{2.1lemma}$ and Equations $\eqref{2.3}$ and $\eqref{3.10}$.
 \end{proof}

The following well known lemma will be used in the sequel. The equivalence of assertions $(i)$ and $(ii)$ of the theorem was proved by Dedekind (cf. \cite[Theorem 6.1.4]{HC}, \cite{RD}). A simple proof of the equivalence of $(ii)$ and $(iii)$ is given in \cite[Lemma 2.1]{JNT}. 

\begin{lemma}\label{dedekind}
	Let  $f(x) \in \Z[x]$ be a monic irreducible polynomial  having the factorization $\bar{g}_1(x)^{e_1} \cdots \bar{g}_{t}(x)^{e_{t}}$ modulo a prime $p$ as a product of powers of distinct irreducible polynomials over $\Z/p\Z$ with $g_i(x) \in \Z[x]$ monic. Let $K=\Q(\theta)$ with $\theta$ a root of $f(x)$.   Then the following statements are equivalent:
	\begin{itemize}
		\item[(i)] $p$ does not divide $[\Z_k:\Z[\theta]]$.
		\item[(ii)] For each $i$, we have either $e_i =1 $ or $\overline g_i (x)$ does not divide $\overline M(x)$ where $M(x) = \frac{1}{p}(f(x) -  g_1 (x)^{e_1} \cdots  g_{t} (x)^{e_{t}} )$.
		\item[(iii)] $f(x)$ does not belong to the ideal $\langle p, g_i(x)\rangle^2$ in $\Z[x]$ for any $i$, $1\leq i\leq t$.
	\end{itemize}
		
\end{lemma}

\section{Proof of Theorem \ref{1.1}}
 \begin{proof}[Proof of Theorem \ref{1.1}]
Let $p$ be a prime number dividing $D_F$. In view of Lemma \ref{dedekind}, $p$ does not divide $[\Z_K : \Z[\theta]]$ if and only if $f(x) \not\in \langle p, g(x) \rangle^2$ for any monic polynomial $g(x) \in \Z[x]$ such that $\bar g(x)$ divides $\bar f(x)$ and $g(x)$ irreducible modulo $p$. Note that $f(x) \not\in \langle p,g(x)\rangle ^2 $ if $\bar g(x) $ is not a repeated factor of $\overline{f}(x) $. We prove the theorem case by case.

\noindent Case (i). Suppose $p\mid a$. In this case, we have either $p\mid b$ or $p\nmid b$. Consider the first possibility when $p\mid b$, then  $F(x) \equiv x^{mn}~(mod~p)$. Clearly $F(x) \in \langle p, x\rangle^2$ if and only if $p^2$ divides $a$, consequently $p\nmid [\Z_K : \Z[\theta]]$ if and only if $p^2\nmid a$. We now consider the second possibility when  $p\nmid b$. In this situation, $ {F}(x)\equiv(x^m-{b})^n$ (mod $p$). Let $m=p^js$ with $j\geq 0$. In view of Binomial theorem, we obtain
\begin{equation}
  F(x)\equiv (x^m-b)^n\equiv (x^s-b)^{np^j} ~(mod ~p).
\end{equation}
Let $\bar{g}_1(x) \cdots \bar{g}_{t}(x)$ be the factorization of $x^{s} - \overline{b}$ over $\Z/p\Z$, where $g_i(x) \in \Z[x]$ are monic polynomials which are distinct and irreducible modulo $p$. Write $x^{s} - b$ as $g_1(x)\cdots g_t(x) + pH(x)$ for some polynomial $H(x) \in \Z[x]$. Using Binomial theorem, one can easily check that 
$$h(x^{p^j}) = h(x)^{p^j}+ph(x)T(x)+b^{p^j}-b$$
for some polynomial $T(x)\in\Z[x]$. Therefore, keeping in view that $F(x)=(h(x^{p^j}))^n-a$, it follows that
\begin{equation}\label{eq;2.2}
	 F(x)=(h(x^{p^j}))^n-a=(h(x)^{p^j}+ph(x)T(x)+b^{p^j}-b)^n-a.
\end{equation}
  Using again Binomial theorem and the fact $p$ divides $b^{p^j}-b$ in Equation $\eqref{eq;2.2}$, we see that
\begin{equation}\label {eq;2.31}
	 F(x)=(h(x)^{p^j})^n+ph(x)M_1(x)+p^2M_2(x)
+(b^{p^j}-b)^n-a, 
\end{equation} for some polynomials $M_1(x),M_2(x)\in\Z[x]$.
Since $n \geq 2$, the first four terms on the right hand side of $\eqref{eq;2.31}$ belong to   $\langle p, g_i(x) \rangle^2$ for each $i$, $1\leq i\leq t$. So $F(x) \in \langle p, g_i(x) \rangle^2$ for some $i$, $1\leq i\leq t$  if and only if $p^2\mid a$. Therefore, by  Lemma \ref{dedekind} ~$p\nmid[\Z_K : \Z[\theta]]$ if and only if $p^2\nmid a$. This proves the theorem in Case (i). \\\\
\noindent Case (ii). Suppose $p \nmid a$ and $p\mid b$. Since $p \mid D_F$, we conclude that $p$ divides $mn $ in view of  $(\ref{eq:1.1})$. Write  $n = s'p^{k}$, $m = sp^{j}$ where $p\nmid{ ss'}, ~j\geq 0,k\geq 0$ along with $j+k\geq 1$. In view of Binomial theorem, we obtain
\begin{equation}
	F(x)\equiv(x^m -b)^n-a\equiv (x^{ss'}-a)^{p^{j+k}} (mod ~p).
\end{equation}
 Let $\bar{g}_1(x) \cdots \bar{g}_{t}(x)$ be the factorization of $x^{ss'} - \overline{a}$ over $\Z/p\Z$, where $g_i(x) \in \Z[x]$ are monic polynomials which are distinct and irreducible modulo $p$. Denote $h(x)=x^{ss'} - a$ and write $h(x)$ as $g_1(x)\cdots g_t(x) + pH_1(x)$ for some polynomial $H_1(x) \in \Z[x]$. Using Binomial theorem  for $(x^m-b+b)^n = (h(x)+a)^{p^{j+k}}$, we see that
 \begin{equation*}
 (x^m-b)^n = (h(x)+a)^{p^{j+k}} -  \left(\sum\limits_{i=1}^{n}\binom{n}{i}(x^{m}-b)^{n-i}b^{i}\right)	
 \end{equation*}
and hence using $p \mid b$, we have
\begin{equation}\label{3.5}
	F(x) = (x^{m}-b)^n - a =  h(x)^{p^{j+k}}+ph(x)M_1(x)+p^2M_2(x)+a^{p^{j+k}}-n(x^m-b)^{n-1}b-a
\end{equation} 
for some polynomials $M_1(x),M_2(x)\in \Z[x]$. The first three summands on the right hand side of $\eqref{3.5}$ belong to   $\langle p, g_i(x) \rangle^2$ for each $i$, $1\leq i\leq t$. So $F(x) \in \langle p, g_i(x) \rangle^2$ for some $i$, $1\leq i\leq t$,  if and only if the polynomials $\frac{1}{p}[a^{p^{j+k}}-a-ng(x)^{n-1}b]$ and $x^{ss'}-a$ have common root. By Lemma \ref{dedekind},  $p\nmid [\Z_K :\Z[\theta]]$ if and only if the polynomials $\frac{1}{p}[a^{p^{j+k}}-a-ng(x)^{n-1}b]$ and $x^{ss'}-a$ are coprime modulo $p$. This  proves the theorem in  Case $(ii)$. Here, one can observe that if $p \mid n$, then $p \nmid [\Z_K:\Z[\theta]]$ if and only if $p^2 \nmid (a^{p^{j+k}}-a) $. \\\\
Case (iii). Suppose $p\nmid ab$ and $p|n$. In this case, we have $\overline{F}(x) = (x^m-\bar{b})^{n}-\overline{a}$~ (mod $p$). Write $m=p^j s$ and $n=p^k s^{\prime}$ with $p \nmid ss^{\prime},~j \geq 0, k \geq 1$. Then $\overline{F}(x)=((x^s-b)^{s^{\prime}} -a)^{p^{j+k}}$. Denote $(x^s-b)^{s^{\prime}} -a$ by $h(x)$ so that $F(x)\equiv (x^{p^js}-b)^{p^ks^{\prime}}-a \equiv h(x)^{p^{j+k}}$ (mod $p$). Let $h(x) \equiv  g_1(x)\cdots g_{t}(x)$ (mod $p$) be the factorization of $h(x)$  into the product of irreducible polynomials modulo $p$ with each $g_{i}(x) \in \Z[x]$ monic. Keeping in mind that $h(x)=(x^s-b)^{s^{\prime}} -a$ and $h(x^{p^j}) \equiv {h}(x)^{p^j} (mod~p)$, we see that
\begin{align*}
	F(x) &= (h(x^{p^j})+a)^{p^k}-a \\
	&= (h(x)^{p^j}+pM_1(x)+a)^{p^k}-a,
\end{align*}
for some $M_1(x) \in \Z[x]$. As $k$ is a positive integer, in view of Binomial theorem we have $F(x)= h(x)^{p^{j+k}} +ph(x)M_2(x) + p^2M_3(x)+a^{p^k}-a $ for some $M_2(x),M_3(x) \in \Z[x]$. Using $h(x) = \prod\limits_{i=1}^{t} g_i(x) + pH(x)$ where $H(x)\in \Z[x]$, we obtain
\begin{align*}
	F(x) &= \left( \prod_{i=1}^{t} g_i(x) + pH(x) \right)^{p^{j+k}} + p \left( \prod_{i=1}^{t} g_i(x) + pH(x) \right) M_2(x) + p^2 M_3(x) + a^{p^k} - a \\
	&= \left(\prod_{i=1}^{t} g_i(x)\right)^{p^{j+k}} + p \prod_{i=1}^{t} g_i(x) M_4(x) + p^2 M_5(x) + a^{p^k} - a,
\end{align*}
for some polynomials $ M_4(x), M_5(x) \in \Z[x] $. Then $F(x) \in \langle p,g_i(x) \rangle^2$ for some $i,~1 \leq i \leq t$ if and only if $p^2 \mid (a^{p^k}-a)$. By Lemma \ref{dedekind}, $p$ does not divide $[\Z_K : \Z[\theta]]$ if and only if $p^2 \nmid (a^{p^k}-a)$. This proves the theorem in the present case.

\noindent Case (iv). Suppose $p \nmid abn$ and $p \mid m$. In this case, we have $\overline{F}(x) = (x^m-\bar{b})^{n}-\overline{a}$. Write $m=p^j s$ with $p \nmid s$ and $j \geq 1$. Then ${F}(x)=((x^s-b)^{n} -a)^{p^{j}}$ (mod $p$). Denote $(x^s-b)^{n} -a$ by $h(x)$ so that $F(x)=h(x^{p^j})=(x^{p^js}-b)^n-a $. Clearly $F(x) \equiv h(x)^{p^{j}}$ (mod $p$). Let $h(x) \equiv  g_1(x)\cdots g_{t}(x)$ (mod $p$) be the factorization of $h(x)$  into a product of irreducible polynomials modulo $p$ with each $g_{i}(x) \in \Z[x]$ monic. In view of Binomial theorem, we have
\begin{align*}
	F(x) &= ((x^s-b+b)^{p^j}-b)^n-a \\
	&= \left( (x^s-b)^{p^j} + \sum_{i=1}^{p^j-1} {p^j \choose i} (x^s-b)^{p^j-i} b^i + b^{p^j}-b \right)^n - a \\
	&= (x^s-b)^{n p^j} + n(x^s-b)^{(n-1)p^j} \left( \sum_{i=1}^{p^j-1} {p^j \choose i} (x^s-b)^{p^j-i} b^i + b^{p^j}-b \right) + p^2 N_1(x) - a \\
	&= (x^s-b)^{n p^j} + n \left( \sum_{i=1}^{p^j-1} {p^j \choose i} (x^s-b)^{np^j-i} b^i + b^{p^j}-b \right) + p^2 N_1(x) - a\\
	&= (h(x)+a)^{p^j} + n \left( \sum_{i=1}^{p^j-1} {p^j \choose i} (x^s-b)^{np^j-i} b^i + b^{p^j}-b \right) + p^2 N_1(x) - a
\end{align*}
for some polynomial $N_1(x) \in \Z[x]$. Now using $h(x)= (x^s-b)^n-a$ and noting that $ \nu_p( {p^j \choose i} )= j- \nu_p(i)$, we obtain
\begin{align*}
	F(x) =& (h(x))^{p^j} + ph(x)N_2(x) + a^{p^j} + p^2N_3(x) \\ 
	&+ n \left( \sum_{i=1}^{p-1} {p^j \choose ip^{j-1}} (x^s-b)^{np^j-ip^{j-1}} b^i + b^{p^j}-b \right) - a
\end{align*}
for some $N_2(x),N_3(x) \in \Z[x]$. Our desired result now follows in view of Lemma \ref{dedekind} $(ii)$.

\noindent Case (v). Now consider the last case when  $p\nmid abmn$. Since $p|D_F$, we have $m\geq 2$ by virtue of ($\ref{eq:1.1}$). Let $\beta$ be a repeated root of $\overline{F}(x) = (x^m-\bar{b})^n- \bar{a}$ in the algebraic closure of $\Z/p\Z$. Then
\begin{equation}\label {eq:131}
	\overline F(\beta)=(\beta^m-\bar b)^n- \bar a = \bar 0; ~~~ \overline F'(\beta)=\bar n(\beta^m - \bar b)^{n-1} \bar {m}\beta^{m-1}= \bar 0.
\end{equation}
Observe that $(\beta^m - b)^{n-1} \not \equiv 0$ (mod $p$), otherwise in view of the first equation of (\ref{eq:131}) $- \bar a = \bar{0}$ which is not possible as $p\nmid a$. Therefore, keeping in mind that $p\nmid mn$, it follows that  $\beta = \bar 0$ is the unique repeated root of $\bar{F}(x)$ in $\Z/p\Z$. Also, as $m\geq 2$, observe that $F(x) \in \langle p, x\rangle^2$ if and only if $ p^2 \mid ((-b)^n-a) $. Hence, by Lemma \ref{dedekind}, $p$ does not divide $[\Z_K:\Z[\theta]]$ if and only if $ p^2 \nmid ((-b)^n-a) $. This completes the proof of the theorem.
\end{proof}

 \medskip
  \vspace{-3mm}

 \end{document}